\documentclass[12pt]{article}
\usepackage{amsmath,latexsym,amssymb,amsthm,enumerate,amsmath,amscd}
\usepackage{amsthm,amsmath,amsfonts,amssymb,amscd,latexsym,stmaryrd}
\usepackage[mathscr]{eucal}
\usepackage{color}
\usepackage[usenames,dvipsnames]{xcolor}
\usepackage{colortbl}
\usepackage[all,cmtip]{xy}
\usepackage{graphicx}
\usepackage{comment}
\usepackage{amstext,amsthm,amsmath,amsfonts,amssymb,amscd,latexsym,txfonts,stmaryrd}
\usepackage{color}
\usepackage{epsfig}
\usepackage[all,cmtip]{xy}
\usepackage{appendix}
\usepackage{hyperref}
\usepackage{tikz}

\newcommand{\Ann}{\operatorname{Ann}}

\newcommand{\Sym}{\operatorname{Sym}}

\newcommand{\Z}{{\mathbb Z}}

\theoremstyle{plain}
\newtheorem{theorem}{Theorem}[section]

\newtheorem{corollary}[theorem]{Corollary}
\newtheorem{lemma}[theorem]{Lemma}
\theoremstyle{definition}
\newtheorem{definition}[theorem]{Definition}

\newtheorem{remark}[theorem]{Remark}
\newtheorem{example}[theorem]{Example}

\newtheorem{question}[theorem]{Question}

\newtheorem{thm}{Theorem}

\newtheorem{rem}[thm]{Remark}

\newtheorem*{remark*}{Remark}

\newtheorem*{ack}{Acknowledgment}

\numberwithin{equation}{section}
\numberwithin{table}{section}
\newtheorem*{problem*}{Problem}
\newtheorem*{question*}{Question}
\newtheorem*{questionrephrased*}{Question 1.2 Rephrased}
\newtheorem*{example*}{Example}
\newtheorem*{watthm*}{Watanabe's Theorem}
\newtheorem*{remarks*}{Remarks}
\newtheorem*{claim*}{Claim}

\newtheorem*{proposition*}{Proposition}
\newtheorem*{lemma*}{Lemma}
\newtheorem*{conjecture*}{Conjecture}

\newtheorem*{JBC*}{Watanabe's Bold Conjecture}

\newtheorem*{UPTP*}{Universal Property of Tensor Products}
\newtheorem*{problem1*}{Problem 1}
\newtheorem*{problem2*}{Problem 2}
\newtheorem*{fact*}{Fact}

\newtheorem*{F1*}{Fact 1}
\newtheorem*{F2*}{Fact 2}
\newtheorem*{F3*}{Fact 3}

\definecolor{purple}{rgb}{0.4,0.2,0.4}

\def\cha{\mathrm{char}\ }

\def\Hom{\mathrm{Hom}}

\def\Hilb{\mathrm{Hilb}}

\def\<{\left<}
\def\>{\right>}
\def\Gl{\mathrm{Gl}}

\def\F{{\sf{k}}}

\def\Sym{\mathrm{Sym}}

\def\ns{\footnotesize \it}

\newcommand{\kk}{{\sf k}}

\def\cha{\mathrm{char}\ }

\begin{document}
\title{Artinian Gorenstein algebras that are free extensions over ${\sf k}[t]/(t^n)$, and Macaulay duality}

\author{
Anthony Iarrobino\\[.05in]
{\ns Department of Mathematics, Northeastern University, Boston, MA 02115,
 USA.}\\[.2in]
 Pedro Macias Marques\\[.05in]
{\ns Departamento de Matem\'{a}tica, Escola de Ci\^{e}ncias e Tecnologia, Centro de Investiga\c{c}\~{a}o}\\[-.05in]
{\ns  em Matem\'{a}tica e Aplica\c{c}\~{o}es, Instituto de Investiga\c{c}\~{a}o e Forma\c{c}\~{a}o Avan\c{c}ada,}\\[-.05in]
{\ns Universidade de \'{E}vora, Rua Rom\~{a}o Ramalho, 59, P--7000--671 \'{E}vora, Portugal.}\\[.2in]
Chris McDaniel\\[0.05in]
{\ns Endicott College, 376 Hale St
Beverly, MA 01915, USA.}\\[.2in]}

\date{July 18, 2018, revised August 9, 2019}
\maketitle

\begin{abstract}
T. Harima and J. Watanabe studied the Lefschetz properties of free extension Artinian algebras $C$ over a base $A$ with fiber $B$. The free extensions are deformations of the usual tensor product; when $C$ is also Gorenstein, so are $A$ and $B$, and it is natural to ask for the relation among the Macaulay dual generators for the algebras.  Writing a dual generator $F$ for $C$ as a homogeneous ``polynomial'' in $T$ and the dual variables for $B$, and given the dual generator for $B$, we give sufficient conditions on $F$ that ensure that $C$ is a free extension of $A={\sf k}[t]/(t^n)$ with fiber $B$. We give examples exploring the sharpness of the statements. We also consider a special set of coinvariant algebras $C$ which are free extensions of $A$, but which do not satisfy the sufficient conditions of our main result.
\end{abstract}
\section{Introduction}
Let ${\F}$ be an arbitrary field, and let $A$ be a commutative $\Z_{\geq 0}$-graded connected Artinian algebra, which we shall simply call a graded Artinian algebra for short.  We say that $A$ has the \emph{standard grading} if it has algebra generators in degree one, i.e. if $A=\F[A_1]$; otherwise it is non-standard graded.  The unique (homogeneous) maximal ideal is $\mathfrak{m}_A=\oplus_{i=1}^jA_i$, and the socle of $A$ is the colon ideal $(0:\mathfrak{m}_A)$. By the \emph{highest socle degree} we shall mean the largest integer $j\geq 0$ for which the graded component $A_j$ is non-zero.  We say that $A$ is a \emph{graded Artinian Gorenstein algebra} if its socle has $\F$-dimension equal to one; in this case $j$ is simply the \emph{socle degree} and we have the equality $A_j=(0:\mathfrak{m}_A)$, and in particular $A_j$ has $\F$-dimension one. 
Our convention is that all
Artinian algebras in this article are graded, possibly non-standard graded, unless  we specify otherwise; also, all maps between graded objects preserve the grading.

The notion of free extension generalizes that of a tensor product, and was introduced by T. Harima and J. Watanabe in \cite{HW1,HW2} to study the strong Lefschetz property of Artinian algebras. In \cite[Theorem 6.1]{HW3} they simplify their proof of a main earlier result: namely, they show that if $A,B,C$ are graded Artinian algebras of symmetric Hilbert functions, and if $C$ is a free extension of $A$ with fiber $B$, then both $A$ and $B$ are strong Lefschetz implies that $C$ is strong Lefschetz (see Remark \ref{Lefrem} below).\par 
Let $A$, $B$, and $C$ be graded (not necessarily standard) Artinian algebras, with maps $\iota\colon A\rightarrow C$ and $\pi\colon C\rightarrow B$. 
\begin{definition} \cite[\S4.2-4.4]{H-W}.
	\label{def:FreeExt}
	The Artinian algebra $C$ is a \emph{free extension} with base $A$ and fiber $B$ if both
	\begin{enumerate}[i.]
		\item The inclusion $\iota\colon A\rightarrow C$ makes $C$ into a free $A$-module.
		\item The projection $\pi\colon C\rightarrow B$ is surjective with $\ker(\pi)=(\iota(\mathfrak{m}_A))\cdot C$.
	\end{enumerate}
\end{definition}

When $A,B$ are Artinian Gorenstein, a free extension $C$ of $A$ with fiber $B$ is also Artinian Gorenstein (Lemma \ref{lem:ABCGor}).  It is well known that a graded Artinian Gorenstein (AG) algebra $R/I$ is completely determined by a single homogeneous polynomial in the dualizing module of $R$, called its Macaulay dual generator (Lemma~\ref{lem:GorensteinMDPD}).  We focus in this paper on the following question.
\begin{question}
	\label{ques:3}
	Suppose that $A$, $B$, and $C$ are graded Artinian Gorenstein, and $C$ is a free extension with base $A$ and fiber $B$.  What is the relationship between the Macaulay dual generators of $A$, $B$, and $C$?  For example, if $A$ and $B$ are graded Artinian Gorenstein algebras with Macaulay dual generators $F_A$ and $F_B$, respectively, can we construct using $F_A,F_B$ and some further information, a Macualay dual generator $F_C$ for a free extension $C$ of $A$ with fiber $B$?  
\end{question} 
We study graded Artinian Gorenstein algebras $C$ that are free extensions over  $A={\F}[t]/(t^n)$ with graded Artinian Gorenstein fibers $B$: we explicitly relate the dual generators of $A,B,C$, so answer the question in this special case. Let $R={\sf k}[x_1,\ldots,x_r]$ and $S=R[t]$ be polynomial rings with dual rings $Q_R$ and $Q_S$, respectively.  Our main result, Theorem \ref{thm:Main}, gives sufficient conditions for a Gorenstein algebra $C=S/I$ defined by a dual generator $F_C=T^{[n-1]}F_B+T^{[n-2]}G_1+\cdots +T^{[n-1-i]}G_i+\cdots +G_{n-1}  \in  Q_S$ with coefficients $F_B,G_i\in Q_R$, to be a free extension over $A$ with fiber the graded Artinian Gorenstein algebra $B$ having dual generator $F_B$. We show as Corollary~\ref{specialdualgen-thm} the Lemma 3.7 of \cite{IMM1} giving necessary and sufficient conditions for $T^{[n]}F_B+G, G\in Q_R$ to determine a free extension $C$ of $A={\sf k}[t]/(t^n)$ with fiber $B$. Some of the Examples \ref{McD2-ex}-\ref{ex:4} exemplify and apply the theorem, while others show the limitations of the hypotheses. In Section \ref{PBIsec} we for comparison describe a seemingly related but quite distinct notion of Projective Bundle Ideal defined by L. Smith and R.E. Stong in \cite{SmSt-PBI}. In Section \ref{invthsec} we give examples of free extensions $C$ of $A={\sf k}[t]/(t^n)$ defined by a polynomial $F_C$ as above arising from invariant theory,
which do not satisfy the hypotheses of Theorem \ref{thm:Main}.

\tableofcontents
\subsection{Macaulay Duality.}\label{MacDualsec}
Let $A$ be a graded Artinian algebra.  By the \emph{socle} of $A$ we mean the ideal $$(0:\mathfrak{m}_A)=\{x\in A \text { such that }a\cdot x=0, \ \forall \ a\in\mathfrak{m}_A\}.$$  The \emph{type} of $A$ is the ${\F}$-dimension of its socle;  if $A$ has type $1$, it is \emph{Gorenstein} Artinian. For an arbitrary Artinian algebra $A$, recall that the \emph{socle degree} $j_A$ of $A$ is the largest integer $j$ for which $A_j\neq 0$.  Let $A=R/I$ where $R={\F}[x_1,\ldots,x_n]$ is a graded polynomial ring (not necessarily standard) and $I\subset R$ is a homogeneous ideal of finite colength.  Define another graded polynomial ring $Q_R={\F}[X_1,\ldots,X_n]$ with a grading defined by $\deg(X_i)=-\deg(x_i)$ for each $i$.  We regard $Q=Q_R$ as a graded $R$-module where $x_i$ acts on $F=F(X_1,\ldots,X_n)$ via the partial differentiation operator $\partial/\partial x_i$.  If $\cha {\sf k}$ is finite and less or equal $j_A$, the highest socle degree of $A$, we take instead for the dualizing ring $Q={\sf k}_{DP}[X_1,\ldots X_r]$ the ring of divided powers with generators $\bigl\{X_i^{[d]}, 1\le i\le r, 1\le d\bigr\}$ and $X_k^{[i]}\cdot X_k^{[j]}={\binom{i+j}{j}}X^{[i+j]}$; and the contraction action of $R$ on $Q$ induced by $x_i^s\circ X_j^{[k]}=\delta_{i,j}X_i^{[k-s]}$ for $k\ge s$ and zero otherwise  (see \cite[Appendix A]{IK}). We will in fact make this choice, which leads to simpler expressions for $F$. So $Q=Q_R={\sf k}_{DP}[X_1,\ldots,X_r]$. We write $g\circ F$ for the polynomial in $Q$ that results from $g\in R$ acting on $F\in Q$ as contraction. Given polynomials $F_1,\ldots,F_k\in Q$, we define
\begin{equation}
\Ann_R(F_1,\ldots,F_k)=\{g\in R \mid g\circ F_i= 0\}.
\end{equation}
Clearly $\Ann_R(F_1,\ldots,F_k)$ is an ideal in $R$:  it is the annihilator ideal of the $R$-submodule in $Q$ whose generator-set is the span of $F_1,\ldots,F_k$. Given an ideal $I\subset R$ such that the quotient $A=R/I$ is Artinian, we denote by 
$I^\perp$, the $R$-submodule of $Q$
\begin{equation}\label{Iperpeq}
I^\perp=\{F\in Q \mid  I\circ F=0\}.
\end{equation}
\begin{lemma}\cite{Mac}\label{Mac1lem}  Let $\sf k$ be an arbitrary field. 
Then there is a 1-1 correspondence between\begin{enumerate}[i.]
\item finite length graded $A$-closed submodules $M$ of the dual $Q=\Hom(R,\sf k)$ having highest degree $j$,
\item Artinian graded quotients $R/I$, $I=\Ann_RM$ of highest socle degree $j$;
\end{enumerate}
the correspondence from (i) to (ii) is given by $M\to R/\Ann_RM$ and its converse is $R/I\to M=I^\perp$.
\end{lemma}
When $\cha {\F}>j_A$ or $\cha {\F}=0$ we can obtain the analogous result using the partial differentiation, action of $R$ on $Q={\sf k}[X_1,\ldots, X_r]$ viewed as a polynomial ring.  We say that $A$ is a \emph{Poincar\'e duality algebra} if $A_j\cong {\F}$ where $j=j_A$, and if for every degree $0\leq i\leq j$ the vector space pairing
\begin{equation}\label{Poincaredualeq}
\xymatrixrowsep{.5pc}\xymatrix{A_i\times A_{j-i}\ar[r] & A_j\cong {\F}\\
(\alpha,\beta) \ar@{|->}[r] & \alpha\cdot\beta\\}
\end{equation}
is non-degenerate.\par
The following result of F.H.S. Macaulay is well known (see \cite[p.527]{Ei}, \cite[\S1B]{I1}), \cite[\S 2.3]{IK}, \cite[\S 73]{Mac}, \cite[Propositions I.4.2 and I.5.2]{MeSm}).
\begin{lemma}
	\label{lem:GorensteinMDPD}
	Let $A$ be a graded Artinian algebra.  The following are equivalent:
	\begin{enumerate}[i.]
		\item $A$ is Gorenstein of socle degree $j$.
		\item $A=R/I$ where $I=\Ann_RF$ for some homogeneous $F\in Q$ of degree $j$, unique up to a ${\F}^{\ast}$ multiple.
		\item $A$ is a Poincar\'e duality algebra of socle degree $j$. 
	\end{enumerate}
\end{lemma} 
The polynomial $F\in Q$ in (ii) is called the \emph{Macaulay dual generator} of $A$. \par
\subsection{Free Extensions.}\label{freeextsec}
The notion of free extension generalizes that of a tensor product, \cite[\S4.2-4.4]{H-W}  and may be regarded as a deformation of the tensor product \cite[\S2]{IMM1}.  We recall Definition \ref{def:FreeExt}:
\begin{definition}\label{freeextdef} Let $A$, $B$, and $C$ be graded Artinian algebras, with maps $\iota\colon A\rightarrow C$ and $\pi\colon C\rightarrow B$.  We say that $C$ is a \emph{free extension} over $A$ with fiber $B$ if the following conditions both hold:
	\begin{enumerate}[i.]
		\item $\iota\colon A\rightarrow C$ makes $C$ into a free $A$-module, and
		\item $\pi\colon C\rightarrow B$ is surjective with $\ker(\pi)=\mathfrak{m}_A\cdot C$.
	\end{enumerate}
\end{definition}
A sequence 
		\begin{equation}\label{freeextseq}
		\xymatrix{{\sf k}\ar[r] & A\ar[r]^-\iota & C\ar[r]^-\pi & B\ar[r] & {\sf k}}
		\end{equation}
		 is \emph{coexact}  if $\pi$ is surjective and $\ker(\pi)=\mathfrak{m}_A\cdot C$, the ideal in $C$ generated by the image under $\iota$ of positive degree elements of $A$, (this is just condition (ii) of Definition \ref{freeextdef}). This notion of coexact sequence arose in topology -- see \cite{MoSm}.
		 It is straightforward to show the following (see \cite[\S 3.3]{IMM2} for a proof).
		
\begin{lemma}\label{coexactfreelem}
Let $A, B, C$ be graded Artinian algebras with maps $\iota\colon A\rightarrow C$ and $\pi\colon C\rightarrow B$ and suppose that $\pi$ is surjective.  Then the following are equivalent.
	\begin{enumerate}[(i).]
		\item For every $\F$-linear section ${\sf s}\colon B\rightarrow C$ of $\pi$, the map $\Phi_{\sf s}=\iota\otimes{\sf s}\colon _A\left(A\otimes_{\sf k}B\right)\rightarrow _AC$ is an isomorphism, i.e. $C$ is an $A$-module tensor product.
		\item The sequence \eqref{freeextseq}
		 is coexact and $\iota\colon A\rightarrow C$ is a free extension.
		\item $\iota\colon A\rightarrow C$ is a free extension and $\ker(\pi)=(\iota(\mathfrak{m}_A))\cdot C$.
		\item $\ker(\pi)=(\iota(\mathfrak{m}_A))\cdot C$ and $\dim_{\sf k}C=\dim_{\sf k}A\cdot \dim_{\sf k}B$.
	\end{enumerate}

\end{lemma}

Evidently if $A$, $B$, and $C$ are all graded Artinian Gorenstein, or Poincar\'e duality algebras, there is a condition, seemingly weaker than (iv), which is equivalent to items (i)-(iv) of Lemma \ref{coexactfreelem}.  We use the shorthand notation $a\cdot c$ to mean $\iota(a)\cdot c$.

\begin{lemma}\label{lem:SS} \cite[Lemma VI.4.11]{MeSm}
	Let $\xymatrix{\kk\ar[r] & A\ar[r]^-\iota & C\ar[r]^-\pi & B\ar[r] & \kk}$ be a coexact sequence of Poincar\'{e} duality algebras, such that their socle degrees satisfy $j_C=j_A+j_B$.  Then $C$ is a free $A$-module via $\iota$.
\end{lemma}
\begin{proof}
	Note that coexactness of the sequence $\xymatrix{\kk\ar[r] & A\ar[r]^-\iota & C\ar[r]^-\pi & B\ar[r] & \kk}$ implies that $B\cong C/\mathfrak{m}_A\cdot C$, and hence by Nakayama's lemma any $\F$-linear $\pi$-section ${\sf s}\colon B\rightarrow C$ defines a surjection of $A$-modules $\Phi_{\sf s}=\iota\otimes{\sf s}\colon A\otimes_{\sf k}B\rightarrow C$.  In particular, if $j_C=j_A+j_B$ then the two socle degrees agree which means that $\Phi_{\sf s}$ maps a socle generator of $A\otimes_{\sf k}B$ onto a socle generator of $C$.  By the Poincar\'e duality condition on $C$ and $A\otimes_{\sf k}B$, this implies that $\Phi_{\sf s}$ must be injective as well. 
\end{proof}

The following lemma shows that free extensions preserve the Gorenstein property. We thank the referee for suggesting that we include this fact.
\begin{lemma}
	\label{lem:ABCGor}
	Suppose that $A$, $B$, and $C$ are graded Artinian algebras such that $C$ is a free extension with base $A$ and fiber $B$ via maps $\iota\colon A\rightarrow C$ and $\pi\colon C\rightarrow B$.  If $A$ and $B$ are Gorenstein then so is $C$.
\end{lemma}
\begin{proof}
	Assume that $A$ and $B$ are Gorenstein, and let $c_0\in (0:\mathfrak{m}_C)\subset C$ be a homogeneous socle element.  Fix a homogeneous $A$-basis for $C$, say $c_1,\ldots,c_n$ so that $c_0=a_1c_1+\cdot+a_nc_n$ for some homogeneous $a_1,\ldots,a_n\in A$.  Then for any $a\in\mathfrak{m}_A$ we have $\iota(a)\in\mathfrak{m}_C$, and hence 
	$$a\cdot c=\iota(a)\cdot c=(a\cdot a_1)c_1+\cdots+(a\cdot a_n)c_n=0.$$
	Since the $c_1,\ldots,c_n$ are $A$-linearly independent, we must have $a\cdot a_i=0$ for all $i$, and since this holds for every $a\in\mathfrak{m}_A$ we deduce that every $a_i$ lies in the socle of $A$.  Since $A$ has one dimensional socle, there exists field elements $k_1,\ldots,k_n\in\F$ for which $a_i=k_i\cdot a_0$ where $a_0$ is some fixed socle generator of $A$.  Thus we have 
	$$c_0=a_0\cdot \underbrace{(k_1c_1+\cdots+k_nc_n)}_{d_0}.$$
	We would like to show that $d_0$ is a homogeneous $\pi$-lift of some socle generator of $B$.  To see this set $b_0=\pi(d_0)$, and fix $b\in\mathfrak{m}_B$.  Let $c\in C$ be any (homogeneous) lift of $b$.  Then write $c\cdot d_0=a_1'c_1+\cdots+a_n'c_n$ for some $a_i'\in A$.  Since $c\in\mathfrak{m}_C$ we must have $a_0\cdot c\cdot d_0=c\cdot c_0=0$ from which we deduce that $a_i'\in\mathfrak{m}_A$, and hence that $c\cdot d_0\in\iota(\mathfrak{m}_A)\cdot C=\ker(\pi)$.  Thus $\pi(c\cdot d_0)=b\cdot b_0=0$ hence $b_0\in(0:\mathfrak{m}_B)$.  Therefore we have shown that $c_0=\iota(a_0)\cdot d_0$ where $a_0$ is a socle generator of $A$, and $d_0$ is a $\pi$-lift of $b_0$, a socle generator of $B$.  Since any two $\pi$-lifts of $b_0$ differ by an element of $\ker(\pi)=\mathfrak{m}_A\cdot C$, we deduce that the space of such $c_0$ is indeed one dimensional, hence $C$ is Gorenstein.
\end{proof}

\begin{remark}
	\label{rem:SS}
	A stronger result was established in a 2008 unpublished note ``Coexact Sequences of Poincar\'{e} Duality Algebras'', by L. Smith and R.E. Stong.  In particular they show that if any two of $A$, $B$, and $C$ are Gorenstein, then so is the third.
\end{remark}

The next example shows that the relationship among Macaulay dual generators of the algebras in a free extension is not straightforward: the dual generator of $C$ may not be simply a product of those for $A,B$.
\begin{example}\label{t2freeextex}
Let $R={\F}[x_1,x_2,x_3]$ be the graded polynomial ring in three variables with the standard grading, and let $e_i$ be the $i$-th elementary symmetric polynomial in those variables $x_1,x_2,x_3$.  Let $C=R/(e_1,e_2,e_3)$ the complete intersection of ideal generator degrees $1,2,3$ and Hilbert function $H(C)=(1,2,2,1)$,  let $B=R/I$ with $I=(x_1+x_2,x_1x_2,x_3)$, of Hilbert function $H(B)=(1,1)$ and denote by $\pi\colon C\rightarrow B$ the natural projection map.  Define the polynomial rings $S=R[t]$, $W={\sf k}[y_1,y_2,t]$ with weights ${\sf w}(y_1,y_2,t)=(1,2,1)$ and $A$ satisfying
\begin{align*}
A= &\, {\F}[y_1,y_2,t]/(y_1+t,y_2+y_1t,y_2t), & \\
\cong & \,{\F}[t]/(t^3)
\end{align*}
of Hilbert function $H(A)=(1,1,1)$, and define the map $\iota\colon A\rightarrow C$ by $\iota(y_1)=x_1+x_2$, $\iota(y_2)=x_1x_2$ and $\iota(t)=x_3$.  Evidently, each of $A,B,C$ are complete intersections (so Gorenstein) and $\dim_{\F}(C)=\dim_{\F}(A)\cdot\dim_{\F}(B)$. Also, the ideal $(\iota(\mathfrak{m}_A))C=(\iota(t))C=x_3C$ and is the kernel of $\pi: C\to B$.  Hence the sequence $\xymatrix{{\sf k}\ar[r] & A\ar[r]^-\iota & C\ar[r]^-\pi & B\ar[r] & {\sf k}}$ is coexact.  By Lemma \ref{coexactfreelem}(iv) $C$ is a free extension of the base $A$ with fiber $B$.  Let $Q_R={\F}_{DP}[X_1,X_2,X_3], Q_S={\F}_{DP}[X_1,X_2,X_3,T],$ and $Q_W={\sf k}[Y_1,Y_2,T]$ be the dual rings of $R$, $S$ and $W$, respectively.  We may take as dual generator for $A$ any element $F_A= T^{[2]} +Y_1^{[2]} - Y_1T +Y_2 + \beta (Y_1 -T)$ with $ \beta\in \sf k$. Then the Macaulay dual generators for $A,B,C$, respectively, are\footnote{We could have simply chosen $A={\sf k}[t]/(t^3)$ with $F_A=T^{[2]}$, but have chosen to retain a choice close to the invariant theory origin of the example.}
\begin{align*}
F_A= & \,T^{[2]} +Y_1^{[2]} - Y_1T +Y_2 + \beta (Y_1 -T),\\
F_B= &\, (X_1-X_2),\\
F_C= &\, (X_1-X_2)(X_1-X_3)(X_2-X_3) \text { and, taking } T=X_3,\\
= &\, 2T^{[2]}F_B-T(X_1+X_2)F_B+X_1X_2F_B.
\end{align*}
\end{example}
\noindent
\begin{remark}[Lefschetz Properties]\label{Lefrem} 
Recall that the Jordan type $P_\ell$ of a linear element $\ell$ of an Artinian graded algebra $A$ is the partition giving the Jordan blocks of multiplication by $\ell$ on $A$. Recall that a linear element $\ell$ of a graded Artinian algebra $A$ is termed \emph{strong Lefschetz} (SL) if the multiplication map $\times \ell^d\colon A_i\to A_{i+d}$ has maximal rank for every pair of integers $i$ and $d$  \cite[Definition 3.8]{H-W}.  This is equivalent to saying that the partition $P_\ell$ giving the Jordan type of multiplication by $\ell $ is the conjugate to the partition defined by $H(A)$ \cite[Proposition~3.64]{H-W}\footnote{The assumption that $H(A)$ is unimodal in \cite[Proposition 3.64]{H-W} is not necessary, as shown in our forthcoming \cite[Proposition 2.8]{IMM2}}.  When $H(A)$ is symmetric, e.g. if $A$ is Gorenstein, the Lefschetz condition is also equivalent to the condition that the multiplication maps $\ell^{j-2i}\colon A_i\rightarrow A_{j-i}$ are isomorphisms for each degree $0\leq i\leq j/2$.

\par T. Harima and J. Watanabe \cite[Theorem~6.1]{HW3} (see also \cite[\S4.2-4.4]{H-W} for $\cha {\sf k}=0$, and the authors in \cite[Theorem 2.3]{IMM1} for $\cha {\sf k}=0$ or $\cha {\sf k}\ge j_A+j_B$) showed that if $A,B,C$ are graded Artinian algebras (standard or not) with symmetric Hilbert functions, and if $C$ is a free extension with base $A$ and fiber $B$, then both $A$ and $B$ have the strong Lefschetz property implies that $C$ is strong Lefschetz.
The converse is false, even if $A$, $B$, and $C$ all have the standard grading:  the free extension $C$ being SL need not imply that both the base $A$ or the fiber $B$ are SL.  See  Harima-Watanabe \cite{HW3} for a non-standard graded example, and \cite[Example~3.8]{IMM1}, which is succinctly given here as Example \ref{ex:Junzo}, for a standard graded one. \par

An element $\ell\in \mathfrak m_A=\sum_{i=1}^{j_A}A_i$ -- that is, $\ell$, which may be nonhomogeneous, has order at least one -- is said to have \emph{strong Lefschetz Jordan type} (SLJT) if $P_\ell=H(A)^\vee$. It is open whether $A,B$ both having elements of strong Lefschetz Jordan type implies that the free extension $C$ does, when $H(A)$ and $H(B)$ are in addition symmetric \cite[Question 2.17]{IMM1}. \par
When $A={\sf k}[t]/(t^n)$ is standard graded, it is SL; when $A$ is not standard graded then $t$ is an element of strong Lefschetz Jordan type (SLJT) and $H(A)$ is symmetric.  Many of our examples have $B$ standard graded of embedding dimension two, so then $B$ is always SL provided $\cha {\sf k}=0$ or is greater than the socle degree $j_B$ (a result due to J. Brian\c{c}on's standard bases for an ideal $I$ in ${\sf k}\{x,y\}$ : see \cite[Lemma 2.13]{IMM2} for references). In these cases, the free extension $C$ of the standard graded algebra $A={\sf k}[t]/(t^n)$ with fiber $B$ is also SL.  
\end{remark}

\vskip 0.2cm\noindent
{\bf Example.} Take $F=X_1^{[2]}+X_2^{[2]}$, then $I=\Ann F=(x_1x_2,x_1^2-x_2^2)\subset R$; and $A=R/I\cong_{\F} \langle 1,x_1,x_2,x_1^2\rangle$, of Hilbert function $H(A)=(1,2,1)$ and dualizing module
$R\circ F=\langle  1,X_1,X_2, F\rangle$.  If $\cha {\sf k}\not=2$, let $\ell=x_1+x_2$: then we may decompose $A$ as ${\sf k}[\ell]$-module $A=S_1\oplus S_2$, where  $S_1=\{1,\ell,\ell^2\}$ and $S_2=\{x_1-x_2\}$; so the Jordan type partition $P_\ell=(3,1)=H(A)^\vee$, and $\ell$ is strong Lefschetz. When $\cha{\sf k}=2$ then $\ell^2=0$, and $P_\ell=(2,2)$: then $A$ is not strong Lefschetz.
 \section{Free extensions with base or fiber ${\sf k}[t]/(t^n)$.}\label{specialfibersec}
 Let $R={\sf k}[x_1,\ldots,x_r]$, $S=R[t]=R\otimes_{\sf k}\F[t]=\F[x_1,\ldots,x_r,t]$ be graded polynomial rings and denote by $Q_R,Q_S$ their duals.  Fix a homogeneous element $F\in Q_S$ and define the graded Artinian Gorenstein algebra $C=S/\Ann F$.  We look for conditions on $F$ which make $C$ into a free extension with graded Artinian Gorenstein base $A$ and fiber $B$, in which either $A$ or $B$ is $\F[t]/(t^n)$.
 
 In Section \ref{mainsec} we state and prove our main result, Theorem~\ref{thm:Main}, which gives sufficient and weak necessary conditions on the coefficients in $Q_R$ of $F$, for $C$ to be a free extension with base $A=\F[t]/(t^n)$ (and some fiber $B$). We also give examples that illustrate the sharpness, or lack thereof, of the conditions in the Theorem.\par
 In Section \ref{PBIsec} we resume for comparison an ostensibly similar but quite different construction of L. Smith and R.E. Stong of projective bundle ideals \cite{SmSt-PBI}, in which $C$ is a free extension with fiber $B=\F[t]/(t^n)$ (over some base $A$).  In Section \ref{invthsec} we give more examples of free extensions $C$ with base $A={\sf k}[t]/(t^n)$ coming from invariant theory.  Here $A$ will be the ring of relative coinvariants for a pair of finite subgroups $K\subset W\subset\Gl(V)$, and $B,C$ rings of coinvariants for $K$ and $W$ respectively. These examples do not satisfy the sufficiency conditions of Theorem \ref{thm:Main}, and they are typically non-standard graded.

 \subsection{Artinian Gorenstein free extensions with base $A={\sf k}[t](t^n)$.}\label{mainsec}
 Recall our definitions.  Let $A=\F[t]/(t^n)=\F[t]/\Ann T^{[n-1]}$.  Let $R={\F}[x_1,\ldots,x_r]$ be a graded polynomial ring over an arbitrary field, let
 $I_B\subset R$ be a homogeneous ideal of finite colength such that the quotient $B=R/I_B$ is a graded Artinian Gorenstein algebra of socle degree $j_B$. Let $Q_R={\F}[X_1,\ldots,X_r]$ be the dualizing module of $R$ as above, and let $F_B\in (Q_R)_{j_B}$ be a Macaulay dual generator for $B$.  Let $S={\F}[x_1,\ldots,x_r,t]=R[t]$ and $Q_S={\F}_{DP}[X_1,\ldots,X_r,T]$.  

Let $G_0=F_B$ and let $G_i\in Q_R$ for $i=1,2,\ldots ,n-1$ be homogeneous elements of $Q_R$ such that the element $F\in Q_S$ defined by
 \begin{equation}\label{FGieqn}
 F=T^{[n-1]}G_0+T^{[n-2]}G_1+\cdots +T^{[n-1-i]}G_i+\cdots +G_{n-1}  \in  Q_S.
 \end{equation} 
 is homogeneous as well.
The natural inclusion ${\sf k}[t]\rightarrow S$ passes to an inclusion $\iota\colon A\rightarrow C$, and the natural projection $R[t]\rightarrow R$, $x_i\mapsto x_i, \ t\mapsto 0$ passes to a surjective map $\pi\colon C\rightarrow B$. We have a natural sequence
\begin{equation}\label{natseqeq}
\xymatrix{\kk\ar[r] & A\ar[r]^-\iota & C\ar[r]^-\pi & B\ar[r] & \kk}.
\end{equation}
\noindent
{\bf Problem.} Given $G_0=F_B$,
	find necessary and sufficient conditions on the forms $G_1,\ldots, G_{n-1}$ above such that the algebra $C=S/\Ann F$ defined in \eqref{FGieqn}  is a free extension of $A={\sf k}[t]/(t^n)$ with fiber $B=R/\Ann F_B$. Equivalently, so that the sequence \eqref{natseqeq} is coexact and $|C|=\dim_{\sf k}C$ satisfies $|C|=|A|\cdot |B|$.\par
The following Theorem gives a sufficient condition for $C=S/\Ann F$ to be a free extension, but Example \ref{ex:3} and the examples of Section \ref{invthsec} show  it is not necessary.  We also give a weak necessary condition. 
\begin{theorem}
	\label{thm:Main} Given $A ={\F}[t]/(t^n)$, and $ B=R/\Ann F_B$ as above and the dual forms $G_0= F_B$ and $G_1,\ldots ,G_{n-1}\in Q_R$ and $F$ as in \eqref{FGieqn}, define the sequence of nested ideals $I_0\subseteq\cdots\subseteq I_{n-1}\subseteq R$ 
	\begin{equation}\label{Ieqn}
	I_0=\Ann G_0 \text { and }I_i=(I_{i-1}:\Ann G_i) \text { for }i\in [1,n-1].
	\end{equation}
	  Assume that
\begin{equation}\label{mainthmeqn}
I_i\circ G_{n-1-i}\subseteq R\circ F_B, \ \text{for} \ i=0,\ldots,n-1.
\end{equation} 
	Then the sequence \eqref{natseqeq} is coexact, and $C=S/\Ann F$ is a free extension of the base $A$ with fiber $B$.
	\par
	 Assume conversely that $F$ as in  \eqref{FGieqn} defines a free extension of $A$ with fiber $B$. 
	  Then we must have $I_0\circ G_i\subseteq \sum_{j=0}^{i-1}R\circ G_j$ for $i\in [1,n-1]$.

\end{theorem}
In order to prove the theorem we need several Lemmas. 

\begin{lemma}
	\label{lem:AnnFB} The algebra
	$C=S/\Ann F$ of \eqref{FGieqn} is a free extension with base $A$ and fiber $B$ if and only if for every element $g_{n-1}\in\Ann F_B$ there exist elements $g_0,\ldots,g_{n-2}\in R$ such that 
	\begin{equation}\label{grecursiveeq}
	t^{n-1}g_0+\cdots+tg_{n-2}+g_{n-1}\in\Ann F.
	\end{equation}
\end{lemma}
\begin{proof}
	Since the condition of Lemma \ref{lem:SS} on socle degrees is satisfied,  $C$ is a free extension of $A$ if and only if $\ker(\pi)=(t)\cdot C$ which is equivalent to the requirement that every element $\bar{g}\in\ker(\pi)$ has a representative $g\in (t)\cdot R\subset R$; this is in turn equivalent to the given condition.
\end{proof}\par
Note that the condition that $g=t^{n-1}g_0+\cdots+g_{n-1}\in\Ann F$ of Equation \eqref{grecursiveeq} is equivalent to the condition that 
\begin{equation}
\label{eq:Jennie}
\sum_{k=0}^{n-1-i} g_{k+i}\circ G_{k}=0, \ \ \ i=0,\ldots,n-2
\end{equation}
\begin{lemma}
	\label{lem:Id} 
	Fix an integer $d\in [0,n-1]$, and fix an element $a\in R$.  If $a\circ G_0=\sum_{k=1}^db_k\circ G_k$ for some elements $b_k\in R$, then $a\in I_d$ of Equation \eqref{Ieqn}. 
\end{lemma}
\begin{proof}
	By induction on $d$. If $d=0$ this states that $a\circ G_0=0$ implies that $a\in I_0=\Ann G_0$.  Assume the implication holds for an integer $d<n-1$, and suppose that $a\circ G_0=\sum_{k=1}^{d+1}b_k\circ G_k$ for some set of $b_k\in R$.  Fix $y\in\Ann G_{d+1}$.  Then $(y\cdot a)\circ G_0=\sum_{k=1}^{d}(y\cdot b_k)\circ G_k$ (note $y$ is $0$ on the $G_{d+1}$-term).  By the induction hypothesis, this implies that $y\cdot a\in I_{d}$.  Since this holds for every $y\in\Ann G_{d+1}$, we must have $a\in I_{d+1}=(I_{d}:\Ann G_{d+1})$.
\end{proof}\par
The converse to Lemma \ref{lem:Id} is false, see Examples \ref{ex:3} and \ref{ex:4}.

\begin{proof}[Proof of Theorem \ref{thm:Main}] We first assume the conditions \eqref{mainthmeqn} and will show that they suffice for $C$ to be a free extension.
	By Lemma \ref{lem:AnnFB}, it suffices to show that for each fixed $g_{n-1}\in\Ann(G_0=F_B)$, the Equation \eqref{eq:Jennie} has a solution $g_0,\ldots,g_{n-2}$.  We solve the Equation \eqref{eq:Jennie} downward inductively for $i\in [0,n-2]$. For $i=n-2$, we have the equation $g_{n-2}\circ G_0+g_{n-1}\circ G_1=0$.  Note that $g_{n-1}\in I_0\subseteq I_{n-2}$ hence by our condition $I_{n-2}\circ G_1\subseteq R\circ F_B=R\circ G_0$ we conclude that $g_{n-2}$ exists.  Moreover, by Lemma~\ref{lem:Id} we must also have $g_{n-2}\in I_1=(\Ann G_0:\Ann G_1)$.  Inductively, assume that we have found solutions $g_{i+1},\ldots,g_{n-2}$ with $g_k\in I_{n-1-k}$ for each $k\ge i+1$.  Then the $i$-th equation \eqref{eq:Jennie} for $i$ yields
	\[-g_i\circ F_B=g_{i+1}\circ G_1+\cdots+g_{n-1}\circ G_{n-1-i}.\]
	Since $g_{k+i}\in I_{n-1-k-i}\subseteq I_{n-1-k}$, and our assumption is $I_{n-1-k}\circ G_k\subseteq R\circ F_B$, we see that $g_i$ also exists, and by Lemma \ref{lem:Id} we have $g_i\in I_{n-1-i}$.
	This completes the induction step and the proof of the forward implication.\par
	 Assume conversely that $F$ as in  \eqref{FGieqn} defines a free extension of $A$ with fiber $B$. 
	  Then by Lemma \ref{lem:AnnFB} and  Equation \ref{eq:Jennie} we must have $I_0\circ G_i\subseteq \sum_{j=0}^{i-1}R\circ G_j$ for $i\in [1,n-1]$. This completes the proof of the Theorem.
\end{proof}\par
We obtain the following corollary, which is \cite[Lemma 3.7]{IMM1}.
\begin{corollary}\label{specialdualgen-thm}
Let $A=\F[t]/(t^n)$, $B=R/I_B$, $I_B=\Ann F_B$ be as in Theorem \ref{thm:Main} and  let $G\in Q_R$ be a homogeneous polynomial such that the element $F\in Q_S$ defined by
\[
F=T^{[n-1]}\cdot F_B+G
\]
is homogeneous.
Then $F$ is the Macaulay dual generator of a free extension $C=S/\Ann F$ with base $A$ and fiber $B$ if and only if $I_B^{\,2}\circ G = 0$.
\end{corollary}
\begin{proof}
Assume first that $I_B^2\circ G=0$.  Then $I_B\circ G\subseteq R\circ F_B$ and this is exactly Condition \eqref{mainthmeqn} of Theorem \ref{thm:Main}; here $G_0=F_B$, $G_1=\cdots = G_{n-2}=0$, $G_{n-1}=G$, and $I_1=\cdots=I_{n-2}=I_0=I_B=\Ann F_B$.  Hence by Theorem \ref{thm:Main} it follows that $C=S/\Ann F$ is a free extension over $A=\F[t]/(t^n)$ with fiber $B=R/\Ann F_B$.

Conversely assume that $C=S/\Ann F$ is a free extension over $A=\F[t]/(t^n)$ with fiber $B=R/I_B$ where $I_B=\Ann F_B$.  Then by Theorem \ref{thm:Main} we must have  $I_0\circ G_i\subseteq\sum_{j=0}^{i-1}R\circ G_j$ for each $i=0,\ldots,n-1$, and hence for $i=n-1$ we have $I_B\circ G\subseteq R\circ F_B$, and hence $I_B^2\circ G=0$ as desired.
\end{proof}
\begin{remarks*}
	\begin{enumerate}[i.]
		\item The $i=n-1$ part of our condition \ref{mainthmeqn} in Theorem \ref{thm:Main} is vacuous:  it states $I_{n-1}\circ G_0\subseteq R\circ F_B$ where $G_0=F_B$.  In particular, in the case, $n=2$ our condition is $I_0\circ G_1\subseteq R\circ F_B$ or $I_0^{\,2}\circ G_1= 0$.
		
		\item If $\Ann G_i\subseteq I_{i-1}$, then $I_i=(I_{i-1}:\Ann G_i)=R$, and hence also $I_{i+1}=R,\ldots,I_{n-1}=R$.  In particular, if $\Ann G_1\subset\Ann F_B$, then $I_0=\Ann F_B$ and $I_i=R$ for $i>0$, and our conditions in Theorem \ref{thm:Main} are that $G_1,\ldots,G_{n-2}$ are each derivatives of $F_B$ which is impossible since their respective degrees must also each be strictly larger than that of $F_B$ (for $F$ to be homogeneous).
	\end{enumerate}
\end{remarks*}

Recall that a Gorenstein algebra $A=R/I$ of socle degree $j$ is \emph{compressed} if $A$ has the termwise maximum possible Hilbert function $H$, given the socle degree $j=j_A$ and embedding dimension $r$. Letting $ r_i= \dim_{\sf k}R_i$ this is equivalent to
\begin{equation}\label{compressedeq}
\text {$A$ is compressed $\Leftrightarrow$ for $i$ satisfying }1\le i\le j,\,\dim_{\sf k}A_i=\min\{r_i, r_{j-i}\}.
\end{equation}

\begin{example}[Examples of free extension as in Theorem \ref{thm:Main}, $n=2,3$]\label{McD2-ex} 
The first two parts (i), (ii), for $n=2$ exemplify the special case of Corollary \ref{specialdualgen-thm}. We explore a weakening of the hypotheses in part (iii) of the Example, and find that $C$ is not a free extension. For $n=3$,  part (iv) shows that $I_B^{\,2}\circ G_i=0$ is not a sufficient condition for freeness. Part (v) assumes the stronger condition of Theorem \ref{thm:Main},  then $C$ is a free extension.  In this example we are always assuming that the grading on $S=\F[x_1,\ldots,x_r,t]$ is standard. 
\begin{enumerate}[(i).]
\item First, let $R={\sf k}[x,y]$, $B={\sf k}[x,y]/J$, $J=(x^2,y^2)=\Ann F_B$, $F_B=XY$, and $B\cong \langle1,x,y, xy \rangle$, $H(B)=(1,2,1)$. Let $F=XYT+X^{[3]}$. Here $G=X^{[3]}$, which is even in $J^\perp$. See Remark \ref{perpendicular-rem} -- when $B$ is compressed of even socle degree, $J^2$ has order $j_B+2$, but $G$ has degree $j_B+1$ so $J^2\circ G=0$ does not restrict $G$. Then, with $S=R[t]$,
\begin{equation}\label{C1eqn}
C=S/I,\quad I=\Ann F=(t^2,ty-x^2,y^2), \quad C\cong_{\sf k} \langle 1,x,y,t,xt,yt,xy,txy\rangle,
\end{equation}
 and $H(C)=(1,3,3,1)$. We let $A={\sf k}[t]/(t^2)$, let $\iota (t)=t\in C$, and define $\pi: C\to B$, $\pi(\alpha t+b)=b$ and consider the section ${\sf s}:B\to C$, ${\sf s}(b)=b$ in the basis above for $C$ of Equation \eqref{C1eqn} (that is ${\sf s}:1\to 1, x\to x,y\to y, xy\to xy$). Then $\mathfrak{m}_A\cdot C=tC=
 \langle t,xt,yt, xyt\rangle=\ker (\pi)$. And we have 
 \[
 C\cong_{\sf k} {\sf s}(B)\oplus t{\sf s}(B)
 \]
as vector spaces, hence $C$ is free over $A$ (this also follows from Lemma \ref{coexactfreelem} since one easily sees that the sequence $\xymatrix{{\sf k}\ar[r] & A\ar[r]^-\iota & C\ar[r]^-\pi & B\ar[r] & {\sf k}}$ is coexact, and 
 $\dim_{\sf k}C=\dim_{\sf k} A\cdot \dim_{\sf k}B$).
 
 \item Next, let $F'=XY^{[4]}T+X^{[3]}Y^{[3]}$, then $F_{B'}=XY^{[4]}$ and $B'=R/J'$, $J'=(x^2,y^5)$ of Hilbert function $H(B')=(1,2,2,2,2,1)$, and $A={\sf k}[t]/(t^2)$. Note that $(J')^2\circ G=(J')^2\circ X^{[3]}Y^{[3]}=0$. We have
 \begin{equation} 
 C'=S/I',\quad I'=\Ann F'=(t^2,ty-x^2,y^5),
 \end{equation}
 We have $H(C')=(1,3,4,4,4,3,1)$ and $\dim_{\sf k}C=20=(\dim_{\sf k}A)\cdot (\dim_{\sf k}B)$, so $C$ is a 
 free extension of $A$, as required by Corollary \ref{specialdualgen-thm}.
\item
 We now try to vary the example to give one without the restriction on $G$ that $(\Ann F_B)^2\circ G=0$. [It does not lead to a free extension]\par
 Let $F''=XY^{[4]}T+X^{[6]}$. Then $B''=R/J'$, $J'=\Ann (XY^{[4]})=(x^2,y^5)$ of Hilbert function  $H(B'')=(1,2,2,2,2,1)$, again $A={\sf k}[t]/(t^2)$. However, 
 \begin{equation}
 C''=S/I'',\quad  I''=(t^2, tx^2,yx^2,ty^4-x^5,y^5),
\end{equation}
and $H(C')=(1,3,5,5,5,3,1)$, again
of length $24$, not $2\cdot 10=(\dim_{\sf k} A)\cdot (\dim_{\sf k}B)$, so $C' $ is not a free extension of $A$.
\item  We now consider $F=XYT^{[2]}+X^{[3]}T+X^{[2]}Y^{[2]}$, which satisfies the simpler condition that $I_B^{\,2}\circ G_i=0$, $i=1,2$; and we find that $C$ is not always a free extension.  Since $B=R/J$, $J=\Ann(XY)=R/(x^2,y^2)$, we have $H(B)=(1,2,1)$,  $A={\sf k}[t]/(t^3)$,  $H(A)=(1,1,1)$ and 
$H(A\otimes B)=(1,2,1,0,0)+(0,1,2,1,0)+(0,0,1,2,1)=(1,3,4,3,1)$ of length 12.
We have 
\begin{equation*}
S_1\circ F=\langle XYT+X^{[3]},YT^{[2]}+X^{[2]}T+XY^{[2]},XT^{[2]}+X^{[2]}Y\rangle,
\end{equation*} 
and it is straightforward to see that $F$ is compressed (as in  \eqref{compressedeq}), of Hilbert function $(1,3,6,3,1)$, that $\dim_{\sf k}S_2\circ F=6$, not $4$, and $\dim_{\sf k}C>(\dim_{\sf k}A)\cdot (\dim_{\sf k}B)$. It follows that $C=S/\Ann F$ is not a free extension of $A$, although $F=T^{[2]}F_B+TG_1+G_2$ with $I_B^{\,2}\circ G_1=0$ and $I_B^{\,2}\circ G_2 =0$.

\item 	We now take $F=T^{[2]}XY+TX^{[3]}+XY^{[3]}$, so $F_B=XY$, $G_1=X^{[3]}$ and $G_2=XY^{[3]}$. Then we have $\Ann(F_B=XY)=(x^2,y^2)$, $\Ann(G_1=X^{[3]})=(x^4,y)$, $\Ann(G_2=XY^{[3]})=(x^2,y^4)$.  The ideals of Equation \eqref{Ieqn} satisfy
 $I_1=(I_0:\Ann G_1)=(x^2,y)$ and $I_2=(I_1:\Ann G_2)=R$.  Then it is easy to see that 
\begin{enumerate}
	\item $I_0\circ G_2=(x^2,y^2)\circ XY^{[3]}\subset R\circ XY$
	\item $I_1\circ G_1=(x^2,y)\circ X^{[3]}\subset R\circ XY$, and 
	\item $I_2\circ G_0=R\circ XY\subseteq R\circ XY$.
\end{enumerate}
Thus the conditions Equation \eqref{mainthmeqn} of Theorem \ref{thm:Main} are satisfied, hence $C=\kk[x,y,t]/\Ann(T^{[2]}XY+TX^{[3]}+XY^{[3]})$ is a free extension over $A=\kk[t]/(t^3)$ with fiber $B=\kk[x,y]/\Ann(XY)=\kk[x,y]/(x^2,y^2)$. Here $H(B)=(1,2,1)$ and $C=S/\Ann F$ where $ \Ann F=(x^2-yt,y^2-t^2,yx^2)$, with $H(C)=(1,3,4,3,1)=H(A)\otimes H(B).$\footnote{Here for integer sequences $H(A)=(a_1,\ldots,a_m)$ and $H(B)=(b_1,\ldots,b_n)$ we use the notation $H(A)\otimes H(B)$ to mean the integer sequence $(c_1,\ldots,c_{n+m})$ defined by $c_i=\sum_{k=0}^ia_kb_{i-k}$.}
\end{enumerate}
\end{example}
\begin{example}
 We now give several examples of free extensions where $A$ and $C$ are not standard-graded, and a last
 example where $B$ and $C$ are not-standard graded.
First, let $F=TXY+X^{[3]}Y,  S={\sf k}[x,y,t]$ of weights ${\sf w}(x,y,t)=(1,1,2)$. Here $A={\sf k}[t]/(t^2)$ of Hilbert function $H(A)=(1,0,1)$, $B=R/I_B$, $I_B=(x^2,y^2)$, of Hilbert function $H(B)=(1,2,1)$ and
the free extension $C=S/\Ann F$, $\Ann F=(y^2, T-x^2, T^2)$, of (non-standard) Hilbert function 
$H(C)=H(A)\otimes H(B)=(1,2,2,2,1)$. Here $I_B^{\,2}\circ G=(x^4,x^2y^2,y^4)\circ G=0$, so the hypotheses of
Corollary \ref{specialdualgen-thm} are satisfied.\par
Second, let $F=TX^{[2]}Y^{[2]}+X^{[5]}Y$, $I_F=\Ann F=(t^2, y^3, ty-x^3)$, $B=R/I_B$, $I_B=(x^3,y^3)$ of Hilbert function $H(B)=(1,2,3,2,1)$. Then $I_B^{\,2}\circ X^{[5]}Y=0$ and $C$ is a free extension of $A$ with fiber $B$ and Hilbert function $H(C)=H(A)\otimes H(B)=(1,2,4,4,4,2,1)$.\par
Third, let $F=TX^{[2[}Y+X^{[5]}Y$, $S={\sf k}[x,y,t]$ of weights ${\sf w}(x,y,t)=(1,1,3)$, $A={\sf k}[t]/(t^2)$ of Hilbert function $H(A)=(1,0,0,1)$.  Here $B=R/I_B$, $I_B=(x^3,y^2)$ and $I_B^{\,2}\circ X^{[5]}Y=0$.
So $C=S/\Ann F$, $\Ann F=(y^2,t-x^3,t^2)$, is a free extension of $A$ having Hilbert function $H(C)=H(A)\otimes H(B)=(1,0,0,1)\otimes (1,2,2,1)=(1,2,2,2,2,2,1)$.\par
Fourth, let $F=TX^{[2]}Y$, and ${\sf w}(t)=4$ and $C=A\otimes B$; then $C$ has non-unimodal Hilbert function $H(A)\otimes H(B)=(1,2,2,1,1,2,2,1)$. Here $G_1\in Q_7$ must be zero to satisfy $I_B^{\,2}\circ G_1=0$.  Note that if ${\sf w}(t)<4$ in this example, then $I_B^2\circ G_1=0$ does indeed have non-trivial solutions in $Q_{3+{\sf w}(t)}$.  In general, one can show that if ${\sf w}(t)=1$ that $I_B^2\circ G=0$ always has non-trivial solutions in $Q_{f_B+1}$; see Remark~\ref{perpendicular-rem}.  \par
Finally let $F=T^{[2]}X^{[2]}+X^{[3]}, C = {\sf k}[t,y]/(t^3,x^3-t^2x^2)$ of weights $ {\sf w}(t,x) = (1,2)$,  of non-unimodal Hilbert function $H(C)=(1,1,2,1,2,1,1)$. Here $C$ is a free extension of
$A={\sf k}[t]/(t^3)$ of Hilbert function $H(A)=(1,1,1)$, with non-standard graded fiber $B= k[x]/(x^3)$ of Hilbert function $H(B)=(1,0,1,0,1).$ 
\end{example}\vskip 0.2cm\noindent 

The next example shows again as in Example \ref{McD2-ex}(iv) that the simpler conditions
\begin{equation}\label{I2eqn}
I_B^{\,2}\circ G_i\subseteq R\circ F_B, \ \text{for} \ i=1,\ldots,n-1,
\end{equation}
which are implied by \eqref{mainthmeqn}, are not sufficient to guarantee that $C$ will be a free extension , when $n\ge 3$. The following three examples are standard-graded.
\begin{example}[$I_B^{\,2}\circ G_i=0$ for all $i$ is not sufficent to make $C$ a free extension]
	\label{ex:1}
	Let  $R=\kk[x,y]$, $Q=\kk[X,Y]$, and set $F_B=X^{[2]}$, $G_1=X^{[2]}Y$, and $G_2=0$ so that 
	\[F=T^{[2]}X^{[2]}+TX^{[2]}Y\]
	and $C=S/\Ann F$, with $\Ann F=(t^2-ty,y^2,x^3)$, $A=\kk[t]/(t^3)$ (so $n=3$), and $B=R/\Ann F_B$.
	Then $I_0=\Ann F_B=(x^3,y)$ and $\Ann G_1=(y^2,x^3)$. So
	\[I_1=(\Ann F_B:\Ann G_1)=R=\kk[x,y].\]
	Here the conditions of Theorem \ref{thm:Main} are not satisfied since $I_1\circ G_1=R\circ X^{[2]}Y\not\subset R\circ F_B=R\circ X^{[2]}$.  Also, 
	\[I_0\circ G_2=0\subset R\circ F_B, \ \text{and} \ I_1\circ G_1=(x^3,y)\circ X^{[2]}Y\subset R\circ F_B=R\circ X^{[2]}.\] 
	In other words, $I_0^{\,2}\circ G_i\equiv 0$ for $i=1,2$.  However, $C$ is not a free extension of $A$ with fiber $B$: we have $A\otimes_\kk B\cong \kk[t,x,y]/(t^3,x^3,y)\cong \kk[t,x]/(t^3,x^3)$ has Hilbert function $H(A\otimes_\kk B)=(1,2,3,2,1)$ of length $9$.  But $C=\kk[x,y,t]/\Ann(T^2X^2+TX^2Y)$ has Hilbert function $H(C)=(1,3,4,3,1)$ of length $12$.  Alternatively, we can appeal to Lemma \ref{lem:AnnFB} by showing that for $y\in\Ann(F_B=X^{[2]})$ there is no element $g=t^2g_0+tg_1+y\in\Ann(F=T^{[2]}X^{[2]}+TX^{[2]}Y)$.  Indeed, note that for degree reasons, we must have $g_0=0$ and $g_1=c$ (a constant).  Then Equation \eqref{eq:Jennie} gives
	\begin{equation}
	\label{eq:ExSyst}
	\begin{array}{rl}
	c\circ X^{[2]}Y= & 0\\
	c\circ X^{[2]}+y\circ X^{[2]}Y= & 0\\
	\end{array}
	\end{equation}
	which clearly has no solution.  Hence Lemma \ref{lem:AnnFB} implies that $C$ cannot be a free extension of $A$ with fiber $B$.   
\end{example}

The next example shows that the conditions \eqref{mainthmeqn} of Theorem \ref{thm:Main} are not \emph{necessary}, for $C=S/\Ann F$ as in \eqref{FGieqn} to be a free extension of $A$ with fiber $B$.

\begin{example}
	\label{ex:3}
	Similar setup to Example \ref{ex:1}.  Set $F_B=X-Y$, $G_1=Y^{[2]}-X^{[2]}$, and $G_2=X^{[2]}Y-XY^{[2]}$ so that
\[
F=T^{[2]}(X-Y)+T(Y^{[2]}-X^{[2]})+X^{[2]}Y-XY^{[2]}.
\]
Then $I_0=\Ann(F_B=X-Y)=(x+y,xy)$, $\Ann(G_1=Y^{[2]}-X^{[2]})=(xy,x^2+y^2)$, and $\Ann(G_2=X^{[2]}Y-XY^{[2]})=(x^3,y^3,x^2y+xy^2,x^2y^2)$.  Note that in this case, we have $\Ann G_1\subset\Ann F_B$ hence $I_j=R$ for $j\geq 1$.  Therefore the (sufficient) conditions of Equation \ref{mainthmeqn}
	in Theorem~\ref{thm:Main} cannot be satisfied, as in Remarks (ii) above, after Corollary \ref{specialdualgen-thm}.\par
	On the other hand, we can verify that $C=\kk[x,y,t]/\Ann F$ is a free extension over $A=\kk[t]/(t^3)$ with fiber $B=\kk[x,y]/\Ann F_B=\kk[x,y]/(x+y,xy)$ (in fact, referring to Example \ref{g11nex}, we see that $C$ is the coinvariant ring of $\mathfrak{S}_3$ fibering over the coinvariant ring $B$ of $\mathfrak{S}_2$).  We may check the freeness of $C$ directly by showing that Equation \eqref{eq:Jennie} has a solution for $g_{n-1}=x+y$ and $g_{n-1}=xy$.  The reader can check that $h_1=t+(x+y)\in\Ann F$ as well as $h_2=-t(x+y)+xy\in\Ann F$, hence Lemma \ref{lem:AnnFB} implies that $C$ is a free extension with base $A$ and fiber $B$.  Here $H(A)=(1,1,1)$, $H(B)=(1,1)$ and $H(C)=(1,2,2,1)$; the ideal $I_C=\Ann F=(x+y,xy-t^2 t^3)$.
Note that $I_B^{\,2}\circ G_2=\langle 1,X-Y\rangle\not=0$.
\end{example}

Here is another example of a free extension $C$  of $A={\sf k}[t]/(t^n)$,  but not satisfying the hypothesis of Theorem \ref{thm:Main} Equation \eqref{mainthmeqn}; it is a variation of Example \ref{ex:1}.
\begin{example}
	\label{ex:4}
	Same setup as in Example \ref{ex:1}, except we take $G_2$ so that Equation~\eqref{eq:ExSyst} has a solution.  To achieve this, we can choose $G_2=X^{[2]}Y^{[2]}$ and
$F=T^{[2]}X^{[2]}+TX^{[2]}Y+X^{[2]}Y^{[2]}.$ The conditions in Theorem \ref{thm:Main} are not satisfied.  On the other hand, the elements $h_y=-t+y\in\Ann F$ and $h_{x^3}=t^2x-txy+x^3\in\Ann F$.  Hence by Lemma \ref{lem:AnnFB}, $C$ must be a free extension of $A$ with fiber $B$.
\end{example} 

\begin{remark}\label{perpendicular-rem}
When $n=2$ so $F=TF_B+G$ the condition on $G$ of Corollary~\ref{specialdualgen-thm} is that  $I_B^{\,2}\circ G=0$. Given the AG algebra $B$ of socle degree $j_B=d$ defined by the Gorenstein ideal $I_B$, such elements $G\in (Q_R)_{d+1}$ can always be found. Recall that the tangent space to the punctual Hilbert scheme at an Artinian local algebra $B=R/I_B$ is $\Hom (I_B, B)$, and that when $B$ is Gorenstein, there is a degree reversing $B$-module isomorphism
\begin{equation}
 \nu:\Hom (I_B,B)\cong_{\F} I_B/I_B^{\,2},
 \end{equation}
  with $\Hom(I_B,B)_0\cong (I_B)_{d}/(I_B^{\,2})_d$.  Since for $i>d$ we have $(I_B)_i=R_i$ we have
\begin{equation}
R_i/(I_B^{\,2})_{i}\cong_{\F}(I_B)_i/(I_B^2)_i\cong \Hom(I_B,B)_{d-i}.
\end{equation}
Thus we have, for the dimension of the vector space of $G$,
\begin{align} 
\dim_{\F}\langle\{G\in R_{d+1}\mid I_B^{\,2}\circ G=0\}\rangle\notag
&=\dim_{\F} \langle ((I_B^{\,2})_{d+1})^\perp\cap R_{d+1}\rangle\\&=\dim_{\F} \Hom(I_B,B)_{-1}.\label{G3eqn}
\end{align}
 Some calculations of the dimension in \eqref{G3eqn} are in \cite[\S 6.2]{IK}, in particular \cite[Lemma 6.21]{IK}. This negative one component of the tangent space contains the  trivial degree $-1$ tangents arising from the $r$ partials 
 $\{\partial/\partial x_1,\ldots ,\partial/\partial x_r\}$ mapping $I_B$ to $B$.  When these are the only $-1$ tangents, then $B$ is not only non-smoothable, but is a generic point of a non-smoothable component of the Hilbert scheme $\Hilb^s(\mathbb A^r)$ where $s=\dim_{\F}B$ (these are called ``elementary'' components). Otherwise the space of possible $G$ has larger dimension than $r$.	  
																	  \end{remark}
																  
	Finally, we recall and develop further an example of U. Perazzo (see \cite[Example 3.8]{IMM1}).
\begin{example} 
	\label{ex:Junzo}
Let $R={\sf k}[x,y,z,u,v]$, $S=R[t]$ be standard graded polynomial algebras with dual algebras $Q_R$ and $Q_S$, and set $A=\F[t]/(t^2)$.  Let $F_B=XU^{[2]}+YUV+ZV^{[2]}$, $I_B=\Ann F_B$, and $B=R/I_B$; $j_B=3$ here.  Then $(I_B)_2=({\sf k}[x,y,z]_2,W)$ where $W=\langle xu-yv=w_1,yu-zv=w_2,xv=w_3,zu=w_4\rangle$, and 
$$(I_B^2)_4 ={\sf k}[x,y,z]_4+{\sf k}[W]_2+W\cdot {\sf k}[x,y,z]_2,$$ of dimension $15+10+\dim (W\cdot  {\sf k}[x,y,z]_2)\le 15+10+24=49$. Since $\dim_{\sf k} R_4=70$ the $\F$-vector space of $G\in (Q_R)_4$ satisfying $I_B^2\circ G=0$ is at least $21$ dimensional.  A check in Macaulay 2 shows that $\dim (W\cdot  {\sf k}[x,y,z]_2)=21$, not $24$, so the space of solutions leads to a 24-dimensional family of free extensions of $A$, with fiber $B$, that include $G=XU^{[2]}+YUV+ZV^{[2]}$.
Setting $F=T\cdot F_B+G$, we see by Corollary \ref{specialdualgen-thm} that $C=S/\Ann F$ is a (standard graded) free extension over $A=\F[t]/(t^2)$ with fiber $B=R/I_B$; their Hilbert functions are $H(A)=(1,1)$, $H(B)=(1,5,5,1)$ and $H(C)=(1,6,10,6,1)$.  Moreover one can readily check that the first Hessian determinant of $F$ does not vanish, i.e. $\operatorname{Hess}_1(F)\neq 0$, while the first Hessian determinant of $F_B$ does vanish, i.e. $\operatorname{Hess}_1(F_B)=0$.  So by the Hessian criterion for the strong Lefschetz property, e.g. \cite[Theorem 3.76]{H-W}, we see that the free extension $C$ is strong Lefschetz, while its fiber $B$ is not.  Since $B$ has a symmetric Hilbert function, this implies also that $B$ does not have an element of strong Lefschetz Jordan type.  The polynomial $F_B$ and its vanishing Hessian were discovered by U. Perazzo in 1900.\footnote{U. Perazzo, Sulle variet\'{a} cubiche la cui hessiana svanisce identicamente. G. Mat. Battaglini
38, 337--354 (1900).}   \vskip 0.2cm
When $F_B\in R_3$ is instead generic, then $\dim ( I_B)_2=10$ implies that $(I_B^2)_4$ has dimension at most 55: ten for the squares of a basis of $(I_B)_2$, and $\binom{10}{2}=45$ for the other products of basis elements. A check in Macaulay 2 shows $\dim (I_B^2)_4=54$, Thus, the solutions to $\dim((I_B)^2_4)^\perp=16$ and the family of free extensions of $A={\sf k}[t]/(t^2)$ with fiber a fixed $B=R/I_B$ where $ I_B=\Ann F_B$ are $16$ dimensional.
\end{example}
\subsection{Projective bundle ideals: free extensions with fiber ${\F}[t]/(t^{k+1})$.}\label{PBIsec}
We compare with ``Projective bundle ideals'' defined by L. Smith and R. E. Stong \cite{SmSt-PBI} where the fiber $B={\sf k}[t]/(t^{k+1})$ and $A=R/\Ann F_B$, the reverse of our assumptions.  They give \cite[p. 611, equation $\circledast$]{SmSt-PBI} the coexact sequence
\[
{\sf k}\to A={\sf k}[V]/J\to {\sf k}[V,t]/I\to {\sf k}[t]/(t^{k+1})\to {\sf k}.
\]
They show \cite[Lemma 1.1]{SmSt-PBI} that $C={\sf k}[V,t]/I$ is a free $A={\sf k}[V]/J$ module (here $J=I\cap {\sf k}[t]$) with basis $\{1,t,\ldots ,t^{k}\}$, so $C$ is a free $A$-module and it is a free extension of $A$ having fiber dimension $k+1$ in the sense of T. Harima and J.~Watanabe; here the fiber is the cohomology ring of $\mathbb P^k$.\par
Note that the kernel of $\iota: A \to C$ should be a principal ideal in the socle of $A$ (this explains the ${\sf k}\to A$). In \cite{SmSt-PBI} for $k=1$, the analogous kernel is the ideal generated by $t^2+\alpha_1t+\alpha_2$. We interpret the dual generator of $C$ from \cite[Theorem 2.5]{SmSt-PBI} in our notation. Let $d=j_A$, and let $\hat{h}(t)=t^d+\hat{h}_1t^{d-1}+\cdots+  \hat{h}_d$ the dualizing form in the Poincar\'{e} duality algebra $C$ itself; then the Macaulay dual generator of $C$ is (in the language of \cite{SmSt-PBI} where the dual generator is written in terms of negative powers of the dual variables)
\[
\theta_I =\overline{h}(T)\circ ({\theta}_J\cdot T^{-(d+k)})={\theta}_JT^{-k} +(\overline{h}_1\circ {\theta}_J)T^{-(k+1)} +\cdots +(\overline{h}_d \circ {\theta}_J)T^{-(k+d)}.
\]
In our notation this corresponds to
\[F=F_BT^{[k]}+h_1\circ F_BT^{[k+1]}+h_2\circ F_BT^{[k+2]}+\cdots +h_d\circ F_BT^{[k+d]} .
\] 
We give two examples the first from \cite{SmSt-PBI}, the second ``related to'' our Example~\ref{t2freeextex}.
\vskip 0.2cm

\begin{example}[Projective Bundle Ideal \cite{SmSt-PBI}]\label{PBI-ex}
Translating their \cite[Theorem~2.5]{SmSt-PBI} to our notation for dual generator (in particular we replace their $X$ by $T$) we take 
\[
F=\theta_I=(X^{[2]}+Y^{[2]})T+YT^{[2]}+T^{[3]}, \,\quad \theta_J=X^{[2]}+Y^{[2]}.
\]
Note that the coefficients of $T^{[2]}$ and of $T^{[3]}$ are derivates of $\theta_J=X^{[2]}+Y^{[2]}$ as required in the PBI construction. Here $d=2$ (the degree of $\theta_J$) is the formal dimension (socle degree) of the base, and the fiber dimension is $k+1=2$, so degree $F=d+k=3$. Then we have 
\begin{equation}\label{A2eq}
 A={\sf k}[x,y]/J,\quad J=\Ann \theta_J=(xy,x^2-y^2), \quad A\cong \bigl\langle \overline{1},\overline{x},\overline{y},\overline{xy}\bigr\rangle,
\end{equation}
and $H(A)=(1,2,1)$. Also, $C$ is a complete intersection
\begin{equation}\label{C2eq}
C={\sf k}[x,y,t]/I,\quad I=\Ann F=(J, t(t-y)),\quad C\cong_{\sf k}\langle 1;x,y,t;xy,tx,ty;txy\rangle,
\end{equation}
and $H(C)=(1,3,3,1)$. We take $B={\sf k}[t]/(t^2)$. Then, taking $\iota: A\to C$ the natural inclusion, and $\pi:C\twoheadrightarrow B$ induced by $\pi(x)=\pi(y)=0$ the sequence
\[
\xymatrix{{\sf k}\ar[r] & A\ar[r]^-\iota & C\ar[r]^-\pi & B\ar[r] & {\sf k}}
\]
is coexact: the kernel of $\pi$ is $(x,y)C=\iota(A)_+C$. Since $\dim_{\sf k}C=\dim_{\sf k}=8=(\dim_{\sf k}A)\cdot (\dim_{\sf k} B)$, we have by Lemma \ref{coexactfreelem} that $C$ is a free extension of $A$ with fiber $B$ (this freeness is also apparent from Equations \eqref{A2eq} and \eqref{C2eq}). In the PBI language $C$ is a projective fiber bundle over the base $A$ with fiber dimension two.	\end{example}	
	We next give a PBI ``analogue'' of  Example \ref{t2freeextex}. 
	\begin{example}\label{PBI2-ex}
	We wish to have a fiber $B'={\F}[t]/(t^3)$ of Hilbert function $H(BÕ)=(1,1,1)$ so the fiber dimension $k+1=3$; and we take  $A'=R/\Ann F_0$ where $F_0=(X_1^{[2]}X_2-X_1X_2^{[2]})$, the last term of the form $F$ from Example \ref{t2freeextex}.
	So we have $A'=R/\bigl(x_1^2+x_1x_2+x_2^2,\, x_1^3\bigr)$ of Hilbert function $H(A')=(1,2,2,1)$. We let
\begin{equation}
	    FÕ= T^{[2]}(X_1^{[2]}X_2-X_1X_2^{[2]})  + T^{[3]}(X_1^{[2]}-X_2^{[2]})+T^{[4]}(X_1-X_2)
	    \end{equation}
leading to $C'=S/\Ann F$ where $\Ann F=\bigl(x_1^2+x_1x_2+x_2^2,\, t^2(x_1+x_2),\, t(x_1^2-x_2^2)\bigr)$, so $C'$ is a complete intersection of generator degrees $(2,3,3)$ and Hilbert function $H(C')=(1,3,5,5,3,1)$.
Here $\pi: C'\to B'$ takes $(x_1,x_2)$ to $0$.
It is evident that the ideal $(\iota (A')_+)C'\subset C'$ is the kernel of $\pi:C'\to B'$, and $H(C')=H(A')\otimes H(B')= (1,2,2,1) \otimes (1,1,1)$, so $C'$ is by Lemma \ref{coexactfreelem} a free extension of $A'$ with fiber $B'$.
	\end{example}	\vskip 0.2cm\noindent
\subsection{Examples from Invariant Theory with base $A={\sf k}[t]/(t^n)$.}\label{invthsec}
We suppose that a pair of finite groups $K\subset W\subset \Gl(n,\F)$ acts on the polynomial ring $R=\Sym(V^*)\cong \F[x_1,\ldots,x_n]$. Their invariant subrings satisfy $R^W\subset R^K\subset R$; then the quotients $R_W=R/(R^W)_+\cdot R$ and $R^K/(R^W)_+\cdot R^K$ are called the coinvariant and relative coinvariant ring, respectively.  One can show that if $K$ and $W$ are generated by reflections, then the coinvariant ring $C=R_W$ of $W$ is a free extension of the relative coinvariant ring $A=R^K_W$ with fiber $B=R_K$, the coinvariant ring of $K$. For further details we refer the reader to the papers \cite{IMM1} and \cite{McDCIM}.
Recall that $G(1,1,n)$ in the Shephard-Todd classification \cite{ShepTodd} is just the symmetric group ${\mathfrak S}_n$ acting on the ring $R={\sf k}[x_1,\ldots,x_n]$ by permuting the variables; the subgroup $G(1,1,n-1)\subset G(1,1,n)$ permutes the first $n-1$ variables.

We will denote by $e_{i,n}=e_i(x_1,\ldots,x_n)$ the elementary symmetric functions in $x_1,\ldots ,x_n$, and we will denote by $\hat{e}_{i,n}$ the elementary symmetric functions in the $m$-th powers $x_1^m,\ldots, x_n^m$.\par
\begin{example}\label{g11nex}[Groups $W=\mathfrak S_n$ and $K=\mathfrak S_{n-1}$]
Let $R={\F}[x_1,...,x_n]$ be the graded polynomial ring in $n$ variables with the standard grading, and consider the coinvariant rings $C=R_W$ of $\mathfrak S_n$ and $B=R_K$ of $\mathfrak S_{n-1}$.  Then $C=R/I_W, I_W=(e_1,e_2,\ldots, e_n)$, a complete intersection (CI) of generator degrees $\{1,2,\ldots, n\}$ and $B=R/(e_{1,n-1},\ldots, e_{n-1,n-1},x_n)$ is a CI of generator degrees $\{1,1,2,\ldots,n-1\}$. Denote by $\pi\colon C\rightarrow B$ the natural projection map.  Define the polynomial ring $S=R[t]$. Then $A=R^K_W$ satisfies, letting $a_i=e_{i,n-1}$ for $1\le i\le n-1$, and $t=x_n$,
\begin{align}
A&={\frac{{\sf k}[e_{1,n-1},\ldots,e_{n-1,n-1},t]}{(e_{1,n},\ldots, e_{n,n})}}\notag\\
&\cong {\frac{ {\sf k}[a_1,\ldots,a_{n-1},t]}{(a_1+t,a_2+ta_1,\ldots ,a_{n-1}+ta_{n-2},a_{n-1}t)}}\cong {\sf k}[t]/(t^n),\label{chaineq}
\end{align}
of Hilbert function $H(A)=(1,1,\ldots, 1_{n-1})$ of length $n$. The last congruence arises from the sequence
\begin{equation*}
a_i\equiv -a_{i-1}t\equiv {-1}^{i-1}t^{i-1}a_1\mod I_W \text { for $ i\in [2,n-1]$ so } e_{n,n}=a_{n-1}t\equiv t^n\mod I_W.
\end{equation*}
Define the map $\iota\colon A\rightarrow C$ by $\iota(a_i)=e_{i,n-1}$ for $1\le i\le n-1 $ and $\iota (t)=x_n$. The ideal $(\iota(\mathfrak m_A))C=(\iota(t))C=x_nC$ and is the kernel of $\pi: C\to B$, so the sequence $
		\xymatrix{{\sf k}\ar[r] & A\ar[r]^-\iota & C\ar[r]^-\pi & B\ar[r] & {\sf k}}$ is coexact.  Since $|C|=n!=n\cdot (n-1)!=|A|\cdot |B|$, by Lemma \ref{coexactfreelem}(iv) $C$ is a free extension of the base $A$ with fiber $B$.  Let $Q_R={\F}_{DP}[X_1,\ldots,X_n]$ and $Q_S={\F}_{DP}[X_1,\ldots ,X_n,T]$ be the dual rings of $R,S$, respectively and denote by $E_{i,n-1}$ the elementary symmetric functions in $X_1,\ldots ,X_{n-1}$. Then the Macaulay dual generators for $A,B,C$, respectively, are
\begin{align}
F_A= &\,\, T^{n-1}\notag\\
F_B= &\prod_{1\le i<j\le n-1} (X_i-X_j)\notag\\
F_C= & \prod_{1\le i<j\le n} (X_i-X_j)=F_B\cdot\prod_{1\le i\le n-1}(X_i-X_n)\cong F_B\cdot \prod_{1\le i\le n-1}(X_i-T)\notag\\
= &\pm\left( T^{n-1	}F_B+ \sum_{i=1}^{n-1}({-1})^iT^{n-1-i}E_{i,n-1}.\right)\label{FCeqn}
\end{align}
We may ignore the $\pm= (-1)^{n-1}$ sign in front of $F_C$ as we take $F_C$ up to non-zero constant multiple.  So up to sign the forms $G_i$ from \eqref{FGieqn} satisfy $G_i=E_{i,n-1}$ for $1\le i\le n-1$.  Since $\Ann G_i \subset \Ann G_{i-1}$, we have $I_0=\Ann F_B$ while $I_1,\ldots, I_{n-1}=R$. Evidently, the sufficient condition $I_i\circ G_{n-1-i}\subset R\circ F_B$ for $i\in [0,n-1]$ of Equation \eqref{mainthmeqn} of Theorem \ref{thm:Main} is not satisfied. 
\end{example}
\begin{remark} Instead of working with the partition $(n-1,1)$ as in Example \ref{g11nex}, we could work with $(n-2,2)$,
and find similar formulas expanding $F_C$ in terms of a polynomial in $T,U$, with coefficients in
${\sf k}_{DP}[X_1,\ldots,X_{n-2}]$.  This might suggest that there are potential generalizations of
Theorem \ref{thm:Main} to such a context, where the base algebra $A$ satisfies $A\cong {\sf k}(t,u)/J$ where $J$ is a suitable ideal:  here $J=(u^{n-1},t^n)$. 
\end{remark}
The following is an example, suggested by Shujian Chen, of an infinite family of relative coinvariant rings that do not have the strong Lefschetz property, but have non-homogeneous elements of strong Lefschetz Jordan type, and where the relative coinvariant ring has the form $A={\sf k}[t]/(t^n)$. We will see that, again, as in Example \ref{g11nex}, these examples arise from Macaulay dual forms $F$ that do not satisfy
the conditions of Theorem \ref{thm:Main}.
\begin{example}[$W=G(m,p,n)$, $K=G(m,p',n)$, $p|p'$]
\label{ex:Sn} 
Assume  $p \ | \ p'\ |m$ and $p < p'$. We let  $B=R_K$, $K=G(m,p',n)$ and $C=R_W$, $W=G(m,p,n)$ be the coinvariant rings, and denote by $A= R^{G(m,p',n)}_{G(m,p,n)}$, the relative coinvariant ring.  Here, since $R_K, R_W$ have the same embedding dimension, $K$ is a non-parabolic subgroup of $W$. The ring  $A$ can be simplified up to isomorphism as follows:
\begin{align*}
{\F}[x_1,\dots,x_n]_{G(m,p,n)}^{G(m,p',n)}&=
{\F}[\hat{e}_{1,n},\dots,\hat{e}_{n-1,n},e_n(x_1^{(m/p')},\dots, x_n^{(m/p')})]/
\\
&\qquad\qquad\bigl(\hat{e}_{1,n},\dots,\hat{e}_{n-1,n},e_n(x_1^{(m/p)},\dots, x_n^{(m/p)})\bigr)
\\
&\cong {\F}[x_1^{(m/p')}\cdots x_n^{(m/p')}]/(x_1^{(m/p)}\cdots x_n^{(m/p)})
\\
&\cong {\F}[t]/(t^{p'/p}),\quad \text { where } t=(x_1^{(m/p')}\cdots x_n^{(m/p'}).
\end{align*}
Here, $A$ is a non-standard graded algebra of Hilbert function 
\begin{equation}
H(A)=(1,0,\dots 0,1_k,0,\ldots, 1_{2k}, 0,\ldots ,0,1_{(s-1)k}) \text{ where $k=n\cdot (m/p')$},
\end{equation}
of length $s=p'/p$.
 Since $H(A)$ has interior zeroes $A$ has no linear SL element so $A$ is not strong Lefschetz; but the element $t$ has strong Lefschetz Jordan type.  Since $A,B,C$ are complete intersections and $|C|=|A|\cdot |B|$ it follows that $C$ is a free extension of $A$ with fiber $B$. We now specify the dual generator $F_C$ in a few simple cases, to illustrate their pattern.
\vskip 0.2cm\noindent
{\bf Case 1.} Take $W=G(2,1,2)$ and $K=G(2,2,2)$ over $\sf k=\mathbb R$. Then $B=R_K={\mathbb R}[x,y]/(x^2+y^2,xy)$ and $C=R_W={\mathbb R}[x,y]/(x^2+y^2,x^2y^2)$, $A={\mathbb R}[t]/(t^2)$ are complete intersections with $\iota\colon A\to C$, $\iota(t)=xy$, and $\pi:C\to B$ the natural projection. The sequence \eqref{natseqeq} is coexact, and $H(C)=(1,2,2,2,1)$, $H(A)=(1,1)$, $H(B)=(1,2,1)$, so $|C|=|A|\cdot |B|$ and by Lemma \ref{coexactfreelem} the algebra $C$ is a free extension of $A$ with fiber $B$. Here $A$ has dual generator $T$, $B$ has dual generator $F_B=X^{[2]}-Y^{[2]}$ and $C={\mathbb R}[x,y,t]/I_C$, $I_C=(x^2+y^2, t-xy, x^2y^2)$ has dual generator
\begin{equation}
F=TF_B+XYF_B.
\end{equation}
Again, since $G_1$ is a multiple of $G_0$, $\Ann_R G_1=(x^2+y^2, x^2y^2)\subset  \Ann G_0=(x^2+y^2,xy)$, we have $I_1=R$ in Equation \eqref{Ieqn} and the condition \eqref{mainthmeqn} of Theorem~\ref{thm:Main} is not satisfied.\ (We considered this relative coinvariant ring in
\cite[Example~3.2]{McDCIM} but there we did not discuss the dual generator for $I_C$ in $Q_S$.)\vskip 0.2cm\noindent
{\bf Case 2.} Take $W=G(3,1,3)$ and $K=G(3,3,3)$ over a field $\sf k$ of characteristic zero. Then $B=R_K={\sf k}[x,y,z]/I_B$, $I_B=(x^3+y^3+z^3,x^3y^3+x^3z^3+y^3z^3,xyz)$ and $C=R_W={\mathbb R}[x,y,z]/I_C$, $I_C=(x^3+y^3+z^3,x^3y^3+x^3z^3+y^3z^3,(xyz)^3)$, and $A={\mathbb R}[t]/(t^3)$ are complete intersections with $\iota\colon A\to C, \iota(t)=xyz$, and $\pi:C\to B$ the natural projection. Evidently, the sequence \eqref{natseqeq} is coexact, and $H(A)=(1,0,0,1,0,0,1)$, $H(B)=(1,3,6,8,9,9,8,6,3,1)={\frac{(1-t^3)^2(1-t^6)}{(1-t)^3}}$ and $H(C)={\frac{(1-t^3)(1-t^6)(1-t^9)}{(1-t)^3}}$  so $162=|C|=3\cdot 54=|A|\cdot |B|$ and by Lemma \ref{coexactfreelem} the algebra $C$ is a free extension of $A$ with fiber $B$. Here $A$ has dual generator $T^{[2]}$, $B$ has dual generator $F_B=(X^{[3]}-Y^{[3]})(X^{[3]}-Z^{[3]})(Y^{[3]}-Z^{[3]})$ and $C=S/(I_C\cap R, t-xyz)$. The dual generator $F\in Q_S$ for $C$ is
\begin{equation}
F=T^{[2]}F_B+TXYZF_B+(X^{[2]}Y^{[2]}Z^{[2]})F_B.
\end{equation}
Again, this example is outside the aegis of Theorem \ref{thm:Main}.\vskip 0.2cm\noindent
{\bf Case 3.} Take $W=G(4,1,2)$, $K=G(4,2,2)$. Then $B=R_K={\sf k}[x,y]/I_B$ where $ I_B=(x^4+y^4,(xy)^2)$ and $C=R_W={\sf k}[x,y]/I_C$ where $I_C=(x^4+y^4,(xy)^4)$. We take $A={\sf k}[t]/(t^2)$ and define $\iota(t)=(xy)^2$. Their Hilbert functions are $H(A)=(1,0,0,0,1), H(B)=(1,2,3,4,3,2,1)$ and $H(C)=(1,2,3,4,4,4,4,4,3,2,1)$. Here in $Q_R$, we have $F_B=X^{[5]}Y-XY^{[5]}$ and $F_C=(XY)^{[2]}F_B$.
Now, letting $S=R[t]$ we have $C\cong {\sf k}[x,y,t]/(I_C, t-x^2y^2)$ and $C=S/\Ann F$ where
$F=TF_B+X^{[2]}Y^{[2]}F_B$.
\vskip 0.2cm\noindent
{\bf Remark.} These examples suggest that there should be a free extension theorem analogous to our Theorem \ref{thm:Main} relevant to this situation, where the dual generator $F$ as in \eqref{FGieqn} for $C$ satisfies $G_{i-1}$ is a derivate of $G_i$ for $i=1,\ldots, n-1$.
\end{example}
\begin{ack}  The authors would like to thank Shujian Chen, who suggested the Example \ref{ex:Sn}, and Junzo Watanabe and Emre Sen, also the referee for helpful comments. The second author was partially supported by CIMA -- Centro de Investiga\c{c}\~{a}o em Matem\'{a}tica e
Aplica\c{c}\~{o}es, Universidade de \'{E}vora, project
PEst-OE/MAT/UI0117/ 2014 (Funda\c{c}\~{a}o para a Ci\^{e}ncia e
Tecnologia).
\end{ack}

\end{document}